\documentclass[11pt]{article}
\usepackage{amsmath,amsthm,hyperref,amssymb}

\def\eric#1{}
\def\linji#1{}
\def\ricardo#1{}
\def\juan#1{}
\def\daniel#1{}

\newcommand{\eee}{{\mathrm{e}}}   % 2.718...
\newcommand{\indicator}[1]{{\mathbf{1}\!\left({#1}\right)}}

\newcommand{\Probability}{\mathrm{Pr}}
\newcommand{\Prob}[1]{{\Probability\left[{#1}\right]}}
\newcommand{\ProbSub}[2]{{\Probability_{#1}\left[{#2}\right]}}
\newcommand{\ProbCond}[2]{{\Probability\left[{#1} \mid {#2} \right]}}

\newcommand{\Expectation}{\mathrm{E}}
\newcommand{\Exp}[1]{{\Expectation\left[{#1}\right]}}

\newcommand{\ExpSub}[2]{{\Expectation_{#1}\left[{#2}\right]}}

\newcommand{\muhbd}[2]{{\mu_{#1,#2}}}

\def\E{{\cal E}}
\def\A{{\textrm{BW}}}

\def\Pa{{\cal P}}

\def\eps{\epsilon}

\def\gap{c_{gap}}
\def\GammaN{{\Gamma_{\omega}}}
\def\BD{{\Gamma}}
\def\hp{{\ell}}
\def\tg{{f}}
\def\hv{{\hat{u}}}
\def\wt{{Z}}

\newcommand{\Trel}{{T_{\rm{relax}}}}
\newcommand{\Tmix}{{T_{\rm{mix}}}}
\newcommand{\BDU}[1]{{\BD^{\textrm{1}}_{#1}}}
\newcommand{\BDL}[1]{{\BD^{\textrm{2}}_{#1}}}

\def\BW{{b_0(\delta)}}
\def\fh{\frac{1.01^{1/b}}{1+\omega}}
\newtheorem{theorem}{Theorem}
\newtheorem{lemma}[theorem]{Lemma}
\newtheorem{proposition}[theorem]{Proposition}

\newtheorem{corollary}[theorem]{Corollary}

\newtheorem{example}[theorem]{Example}

\theoremstyle{remark}
\newtheorem*{remark*}{Remark}

\title{Phase Transition for Glauber Dynamics for Independent Sets on Regular Trees}

\author{Ricardo Restrepo\thanks{School of Mathematics, Georgia
Institute of Technology, Atlanta GA 30332.
Email: restrepo@math.gatech.edu.
Research supported in part by NSF grant CCF-0910584.
}
\thanks{Universidad de Antioquia, Departamento de Matematicas, Medellin, Colombia. Scholarship `200~years'.}
\and
Daniel Stefankovic\thanks{
Department of Computer Science, University of Rochester,
Rochester, NY 14627.  Email: stefanko@cs.rochester.edu.
Research supported in part by NSF grant CCF-0910415.
}
\and
Juan C. Vera\thanks{
Department of Management Sciences,
University of Waterloo, Waterloo ON.
Email: jvera@waterloo.edu.
}
\and
Eric Vigoda\thanks{School of Computer Science, Georgia
Institute of Technology, Atlanta GA 30332.
Email: \{vigoda,ljyang\}@gatech.edu.
Research supported in part by NSF grant CCF-0830298 and CCF-0910584.}
\and
Linji Yang$^\P$}

\date{July 11, 2010}

\begin{document}

\maketitle

\begin{abstract}
We study the effect of boundary conditions on the relaxation time (i.e., inverse
spectral gap) of the Glauber dynamics for the hard-core model on the tree.
The hard-core model is defined on the set of independent sets weighted
by a parameter $\lambda$, called the activity or fugacity.  The Glauber
dynamics is the Markov chain that updates a randomly chosen vertex in each
step.  On the infinite tree with branching factor $b$, the hard-core model
can be equivalently defined as a broadcasting process with a parameter
$\omega$ which is the
positive
solution to $\lambda=\omega(1+\omega)^b$, and
vertices are occupied with probability $\omega/(1+\omega)$ when their
parent is unoccupied.
This broadcasting process
undergoes a phase transition between the so-called reconstruction and non-reconstruction regions at $\omega_r\approx \ln{b}/b$.  Reconstruction has
been of considerable interest recently since it appears to be intimately
connected to the efficiency of local algorithms on locally tree-like graphs,
such as sparse random graphs.

In this paper we show that the relaxation time of the Glauber dynamics
on regular trees $T_h$ of height $h$ with branching factor $b$ and $n$ vertices, undergoes a phase transition around the reconstruction threshold.
In particular, we construct a boundary condition for which the relaxation time slows down at
the reconstruction threshold.
More precisely, for any $\omega \le \ln{b}/b$,
for $T_h$ with any boundary condition,
the relaxation time is $\Omega(n)$ and $O(n^{1+o_b(1)})$.
In contrast, above the reconstruction threshold we show that for every $\delta>0$, for $\omega=(1+\delta)\ln{b}/b$,
the relaxation time on $T_h$ with any boundary condition is
$O(n^{1+\delta + o_b(1)})$, and we construct a boundary condition where
the relaxation time is
$\Omega(n^{1+\delta/2 - o_b(1)})$.  To prove this lower bound in
the reconstruction region we introduce a general technique that
transforms a reconstruction algorithm into a set with poor
conductance.
\end{abstract}

\thispagestyle{empty}

\newpage

\pagenumbering{arabic}

\section{Introduction}
\label{sec:introduction}

There has been much recent interest in possible connections between
equilibrium properties of statistical physics models and
efficiency of local Markov chains for studying these models (see, e.g.,
\cite{BKMP,DSVW,Martinelli-lecturenotes,MSW-ising,MSW-soda,TVVY}).  In this paper we study the
hard-core model and establish
new connections between the so-called reconstruction threshold in statistical
physics with the convergence time of the single-site Markov chain known
as the Glauber dynamics.
The hard-core model was studied in statistical physics as model of a
lattice gas~(see, e.g.,  Sokal \cite{Sokal}), and in operations research as
a model of communication network (see Kelly \cite{Kelly}).  It is a natural combinatorial
problem, corresponding to counting and randomly sampling weighted
independent sets of an input graph $G=(V,E)$.
Let $\Omega=\Omega(G)$ denote the set of independent sets of $G$.
Each set is weighted by an activity (or fugacity) $\lambda>0$.
For $\sigma\in\Omega$, its weight is $\wt(\sigma) = \lambda^{|\sigma|}$ where $|\sigma|$ is the
number of vertices in the set $\sigma$.   The Gibbs measure is defined over $\Omega$ as
$\mu(\sigma) = \wt(\sigma)/Z$ where $Z=\sum_{\sigma\in\Omega} \wt(\sigma)$ is the partition function. \linji{changed weight $w$ to $\wt$ so that it doesn't look like $\omega$.}

This paper studies the hard-core model on trees, in some cases with a boundary
condition.  Let $T_h$ denote the complete tree of height $h$ with branching factor $b$.
For concreteness we are assuming the root has $b$ children, but our results,
of course, easily extend to allow $b+1$ children for the root, the so-called Bethe\ lattice.
Let $n$ denote the number of vertices in $T_h$, and let
$L$ denote the leaves of the tree.
A boundary condition is an assignment $\BD$ to the leaves, where in the
case of the hard-core model, $\BD$ specifies a subset of the leaves $L$ that are in the independent set.\linji{modified this sentence}
Then let $\Omega_\BD = \{\sigma\in\Omega: \sigma(L)=\BD\}$ be the set
of independent sets of $T_h$ that are consistent with~$\BD$, and
the Gibbs measure $\muhbd{h}{\BD}$ is defined with respect to $\Omega_\BD$,
i.e., it is the projection of $\mu$ onto $\Omega_\BD$.

The (heat bath) Glauber dynamics
is a discrete time Markov chain $(X_t)$ for sampling from the Gibbs distribution $\mu$
for a given graph $G=(V,E)$ and activity $\lambda$.
We view $\Omega\subset\{0,1\}^V$ where for $X_t\in\Omega$, $X_t(v)=1$ iff
$v$ is in the independent set.
The transitions $X_t\rightarrow X_{t+1}$ of the Glauber dynamics are defined as:
\begin{itemize}
\item  Choose a vertex $v$ uniformly at random;
\item  For all $w\neq v$ set $X_{t+1}(w) = X_t(w)$;
\item If all of the neighbors of $v$ are unoccupied,
set $X_{t+1}(v) = 1$ with probability $\lambda/(1+\lambda)$,  otherwise set
$X_{t+1}(v) = 0$.
\end{itemize}
When a boundary condition $\BD$ is specified, the state space is restricted
to $\Omega_\BD$.  For the case of the complete tree $T_h$ (possibly with
a boundary condition $\BD$) it is straightforward to verify that the Glauber
dynamics is ergodic\ with unique stationary distribution $\mu_h$ (or $\muhbd{h}{\BD}$
when a boundary condition is specified).
Thus, the Glauber dynamics is a natural algorithmic process for sampling
from the Gibbs distribution.
We study the relaxation time of the dynamics,
which is defined as the inverse of the spectral gap of the transition matrix.
See Section \ref{sec:background} for a more detailed definition of the relaxation time.

The Gibbs distribution describes the equilibrium state of the system,
and the Glauber dynamics is a model of how the physical system
reaches equilibrium \cite{Martinelli-lecturenotes}.  Thus, it is interesting to understand
connections between properties of the equilibrium state (i.e., the Gibbs
distribution) and properties of how the system reaches equilibrium (i.e., the
Glauber dynamics).  Models from statistical physics are designed
to study phase transitions in the equilibrium state.  A
phase transition is said to occur when a small change in the microscopic
parameters of the system (in the case of the hard-core model that corresponds
to $\lambda$) causes a dramatic change in the macroscopic properties
of the system.

A well-studied phase transition is uniqueness/non-uniqueness
of infinite volume Gibbs distributions, these are obtained as a limit of
Gibbs measures for a sequence of boundary conditions as $h\rightarrow\infty$.
For the hard-core model on the complete tree, Kelly \cite{Kelly} showed
that the uniqueness threshold is at
$\lambda_u =b^b/(b-1)^{b+1}$ (namely, uniqueness holds iff $\lambda<\lambda_u$). There are interesting
connections between the uniqueness threshold $\lambda_u$ and the efficiency
of algorithms on general graphs.  In particular, Weitz \cite{Weitz} showed
a deterministic fully-polynomial approximation scheme to estimate the partition
function for any graph with constant maximum degree $b$ for activities $\lambda<\lambda_u$.
Recently, Sly \cite{Sly:hardcore-negative} showed that it is NP-hard
(unless $NP=RP$) to approximate
the partition function for activities $\lambda$ satisfying $\lambda_u<\lambda<\lambda_u+\eps_b$ for some small constant $\eps_b$.

We are interested in the phase transition for reconstruction/non-reconstruction.
This corresponds to extremality of the infinite-volume measure obtained
by the ``free'' boundary condition  where free means no boundary condition \cite{Georgii}.
This measure can be generated by the following broadcast process which
constructs an independent set $\sigma$.
Let $\omega$ be the real positive solution of $\lambda=\omega(1+\omega)^b$.
Consider the infinite complete tree with branching factor $b$, and construct
$\sigma$ as follows.
We first include the root $r$ in $\sigma$ with probability $\omega/(1+\omega)$ and
leave it out with probability $1/(1+\omega)$.  Then for each vertex $v$, once the
state of its parent $p(v)$ is determined, if $p(v)\notin\sigma$ then we add
$v$ into $\sigma$ with probability $\omega/(1+\omega)$ and
leave it out with probability $1/(1+\omega)$; if $p(v)\in\sigma$ then we leave
$v$ out of $\sigma$.  Let $\sigma_h$ denote the configuration of $\sigma$ on level $h$,
and let $\nu_h$ denote the broadcast measure on $T_h$.

Reconstruction addresses whether $\sigma_h$ ``influences'' the configuration at the root $r$.
In words, we first generate $\sigma$ using the broadcasting measure, then we fix $\sigma_h$ and resample a configuration $\tau$ on $T_h$ from the Gibbs distribution $\muhbd{h}{\BD}$ with
boundary condition $\BD=\sigma_h$.  Of course, for finite $h$, the configuration at the root $r$ in $\tau$ has a bias to the initial configuration $\sigma(r)$.
Non-reconstruction is said to hold if the root is unbiased in expectation in the limit $h\rightarrow\infty$.  More precisely, reconstruction holds if and only if:
\begin{equation}
\label{eq:reconstruction-defn}
\lim\limits_{h\rightarrow\infty} \ExpSub{\sigma\sim\nu_h}{\left| \muhbd{h}{\sigma_h}(r\in\tau) - \frac{\omega}{1+\omega}\right|} > 0.
\end{equation}
There are many other equivalent conditions to the above definition of reconstruction, see Mossel \cite{MosselSurvey} for a more extensive survey.
%\linji{
%Mossel did not do the covariance way in his survey paper.
%I would like to conclude the following, however one direction is not clear to be true: From the algorithmic perspective,
%reconstruction holds if and only if the covariance between the actual configuration at the root and
%the output of the maximum likelihood reconstruction algorithm is strictly positive as $h \rightarrow \infty$.}
We refer to the reconstruction threshold as the critical $\omega_r$ such that
for all $\omega<\omega_r$ non-reconstruction holds and for all $\omega>\omega_r$
reconstruction holds.  The existence of the reconstruction threshold follows
from Mossel~\cite[Proposition 20]{Mossel-2nd-eigenvalue}, and, by recent work
of Bhatnagar et al~\cite{BST} and Brightwell and Winkler~\cite{BW}, it is known that
$\omega_r=(\ln{b} + (1+o(1))\ln{\ln{b}})/b$.

Reconstruction for the Ising and Potts models
has applications in phylogenetics \cite{DMR}
and for random constraint satisfaction problems
is connected to the geometry of the space of solutions on sparse random graphs
\cite{ACO,GMON,KMRSZ,MRT}.
Our interest in this paper is on establishing
more detailed connections between the reconstruction threshold and
the relaxation time of the Glauber dynamics for trees.
%\eric{Say more about the idea that non-reconstruction is believed to
%intimately connected
%to the efficiency of certain belief propagation algorithms on sparse
%random graphs or locally tree-like graphs?}
Berger et al \cite{BKMP} proved that for the tree $T_h$ with
boundary condition $\BD$  such that $\muhbd{h}\BD=\nu_h$,
$O(n)$ relaxation time for all $h$ implies non-reconstruction.  For
the Ising model and colorings the boundary condition is empty, i.e.,
$\nu_h$ corresponds to the free boundary condition.   That is
not the case for the hard-core model as discussed further in Section \ref{sec:lower-bound-approach}.

It was recently established for the Ising model \cite{BKMP,MSW-ising,DLP}
and for k-colorings \cite{TVVY} that on the tree $T_h$ with
free boundary condition, the relaxation is $O(n)$ in the non-reconstruction region
and there is a slow down in the reconstruction region.
Our starting point was addressing whether a similar phenomenon occurs in
the hard-core model.  Martinelli et al \cite{MSW-soda} showed that
for the hard-core model on $T_h$ with free boundary condition the
relaxation time is $O(n)$ for all $\lambda$ (and the mixing time is $O(n\log{n}))$.  However, it is unclear
whether the reconstruction threshold has any connection to the relaxation
time of the Glauber dynamics on trees for the hard-core model.  (In fact,
we vacillated between proving that there is fast mixing for all boundary
conditions and proving the following result.)

We prove there is a connection by constructing a boundary condition for
which the relaxation time slows down at the reconstruction threshold.
Here is the formal statement of our results.

\begin{theorem}
\label{thm:main}
For the Glauber dynamics on the hard-core model with activity $\lambda= \omega (1+\omega)^b$ on the complete tree $T_h$ with $n$ vertices,
height $h$ and branching factor $b$, the following hold:
\begin{enumerate}
\item {\bf For all $\omega\le\ln{b}/b$:}\linji{deleted $b_0$ and for all  $\delta > 0$}
\label{thm:below}
\begin{equation*}
\Omega(n)  \le \ \Trel \ \le  O(n^{1+o_b(1)}).
\end{equation*}
\item {\bf For all $\delta>0$ and $\omega=(1+\delta)\ln{b}/b$:}
\begin{enumerate}
\item
\label{thm:above-upper}
For every boundary condition,
\begin{equation*}
\Trel \ \le O(n^{1+\delta+o_b(1)}).
\end{equation*}
\item
\label{thm:above-lower}
There exists a sequence of boundary conditions for all $h\rightarrow\infty$ such that,
\[
\Trel = \Omega(n^{1+\delta/2-o_b(1)}).
\]
\end{enumerate}
\end{enumerate}
\end{theorem}
\begin{remark*}
%The $o_b(1)$ functions are $O(\ln{\ln {b}}/\ln{b})$ in the lower bounds and $O((\ln{\ln {b}})^2/\ln{b})$ in the upper %bounds. \juan{Maybe should just use $O((\ln{\ln {b}})^2/\ln{b})$ in both}\linji{I think it is ok in this way since
%in this version we mainly talk about the lower bound.}
More precisely, we show that there is a function $g(b)=O(\ln\ln{b}/\ln{b})=o(1)$ such that for every $b$,
the lower bound in Part \ref{thm:above-lower} is $\Omega(n^{1+\delta/2-g(b)})$,
and there is a function $f(b) = O((\ln\ln{b})^2/\ln{b})=o(1)$ such that for every $b$,
the upper bound in Part \ref{thm:below} is $O(n^{1+f(b)})$ and in Part \ref{thm:above-upper} is $O(n^{1+\delta+f(b)})$. \linji{changed the remark}
\end{remark*}

The upper bound improves upon Martinelli et al \cite{MSW-soda} who
showed $O(n)$ relaxation time (and $O(n\log{n})$ mixing time)
for $\lambda < 1/(\sqrt{b}-1)$ for all boundary conditions.
Note, $\lambda = 1/\sqrt{b}$ is roughly equivalent to
$\omega \approx \frac{1}{2}\ln{b}/b$ which is
below the reconstruction threshold.
Our main result extends the fast mixing up to the reconstruction threshold, and
shows the slow-down beyond the reconstruction threshold.
Our lower bound in the reconstruction region
uses a general approach that transforms an algorithm showing reconstruction
into a set with poor conductance, which implies the lower bound on the
relaxation time.  This framework captures the proof approach used in
\cite{TVVY}.
%\ricardo{the second order term is $-\frac{\ln\ln\sqrt{b}}{b}$}

In Section \ref{sec:background} we formally define various terms and
present the basic tools used in our proofs.  The lower bound (Part \ref{thm:above-lower} of Theorem
\ref{thm:main}) is presented in Sections \ref{sec:lower-bound-approach},
\ref{sec:lower-bound-broadcasting} and \ref{sec:lower-bound-hardcore}.
Section \ref{sec:lower-bound-approach} outlines the approach.  We then
prove an analogue of Theorem \ref{thm:main} in Section \ref{sec:lower-bound-broadcasting}
for the broadcasting model and use it in Section \ref{sec:lower-bound-hardcore} to prove Part \ref{thm:above-lower} of Theorem \ref{thm:main}.
The argument for the upper bounds
stated in Theorem \ref{thm:main} is presented in Section \ref{sec:upper-bound}. \linji{changed the wording according to Daniel's suggestion}

\section{Background}
\label{sec:background}
Let $P(\cdot,\cdot)$ denote the transition matrix of the Glauber dynamics.
Let $\gamma_1 \ge \gamma_2 \ge \dots\ge \gamma_{|\Omega|}$ be the eigenvalues of the transition matrix $P$.
The spectral gap $\gap$ is defined as $1-\gamma$ where $\gamma=\max\{\gamma_2,|\gamma_{|\Omega|}|\}$ denotes the second largest eigenvalue in absolute value.
The relaxation time $\Trel$ of the Markov chain is then defined as $\gap^{-1}$, the inverse of the spectral gap.
Relaxation time is an important measure of the convergence rate of a Markov chain~(see, e.g., Chapter 12 in \cite{LPW}).

To lower bound the relaxation time we analyze
conductance.  The conductance of a Markov chain with state space $\Omega$ and
transition matrix $P$ is given by $
\Phi=\min_{S\subseteq\Omega} \{\Phi_S\}
$,
where $\Phi_S$ is the conductance of a specific set $S \subseteq \Omega$ defined as
\[
\Phi_S=\frac{\sum_{\sigma\in S}\sum_{\eta\in\bar{S}}\pi(\sigma)P(\sigma,\eta)}
{\pi(S)\pi(\bar{S})}.\]

Thus, a general way to find a good upper bound on the conductance is to find a set~$S$ such that the probability of ``escaping'' from~$S$ is relatively small.
The well-known relationship between the relaxation time and the conductance was established in~\cite{LawlerSokal} and~\cite{SinclairJerrum},
and we will use the form $\Trel = \Omega(1/\Phi)$ for proving the lower bounds.

\section{Lower Bound Approach}
\label{sec:lower-bound-approach}

First note that the lower bound stated
in Part \ref{thm:below} of Theorem \ref{thm:main},
namely, $\Trel=\Omega(n)$, is trivial for all $\omega$.
For example, by considering the set $S=\{\sigma\in\Omega:r\notin\sigma\}$ of
independent sets which do not contain the root, $\Phi(S) = \Omega(1/n)$ since
we need to update $r$ to leave $S$.

We begin by explaining the high level idea of the non-trivial lower bound in
Part \ref{thm:above-lower} of Theorem \ref{thm:main}.  To that end,
we first analyze a variant of the hard-core model in which there are two different
activities, the internal vertices have activity $\lambda$ and
the leaves have activity~$\omega$.  The resulting Gibbs distribution is
identical to the measure~$\nu_h$ defined
in Section \ref{sec:introduction} for the broadcasting process.
Thus we refer to the following model as the broadcasting model.

For the tree $T_h=(V,E)$, we look at the following
equivalent definition of the distribution $\nu_h$ over the set $\Omega$
of independent sets of $T_h$.  For $\sigma\in\Omega$, let
\[  \wt'(\sigma) = \lambda^{|\sigma\cap V\setminus L|}\omega^{|\sigma\cap L|},\]
where $L$ are the leaves of $T_h$ and $\omega$ is, as before,
the positive solution to $\omega(1+\omega)^b=\lambda$.  Let
$\nu_h(\sigma)=\wt'(\sigma)/Z'$ where $Z'=\sum_{\sigma\in\Omega} \wt'(\sigma)$
is the partition function. By simple calculations,
the following proposition holds.

\begin{proposition}
\label{pro:broadcasting-alpha}
The measure $\nu_h$ defined by the hard-core
model with activity $\lambda$ for internal vertices and $\omega$ for leaves
is identical to the measure defined by the broadcasting process.
\end{proposition}
\begin{proof}
In fact, we just need to verify that in the hard-core model with
activity $\lambda$ for internal vertices and $\omega$ for leaves, the
probability $p_v$ of a vertex $v$ being occupied conditioning on its
parent is unoccupied is $\omega/(1+\omega)$.
This can be proved by induction. The base case is $v$ being a leaf,
which is obviously true by the Markovian property of the Gibbs
measure.
If $v$ is not a leaf, by induction, the probability $p_v$ has to
satisfy the following equation
\[
p_v = (1-p_v)\frac{\lambda}{(1+\omega)^b},
\]
which solves to $p_v = \omega/(1+\omega)$.
\end{proof}
The result of Berger et al \cite{BKMP} mentioned in Section \ref{sec:introduction}
implies that the relaxation time of the Glauber dynamics on the
broadcasting model is $\omega(n)$.  We will prove a stronger result,
analogous to the desired lower bound for Part \ref{thm:above-lower}
of Theorem \ref{thm:main}.
\begin{theorem}
\label{thm:broadcasting}
For all $\delta > 0$,
the Glauber dynamics for the broadcasting model on the
complete tree $T_h$ with $n$ vertices, branching factor $b$
and $w=(1+\delta)\ln{b}/b$
satisfies the following:\ricardo{Erased $b>b_0$. Now, it is inside the $o$ term}
\[
\Trel = \Omega(n^{1+\delta/2-o_b(1)}),
\]
where the $o_b(1)$ function is $O(\ln{\ln {b}}/\ln{b})$.
\end{theorem}
\begin{remark*}
We can show a similar upper bound on the relaxation time for the Glauber dynamics in this setting as in Theorem \ref{thm:main}.
Moreover, we can show the same upper bound for the mixing time by establishing a tight bound between the inverse log-Sobolev constant
and the relaxation time as was done for colorings in Tetali et al \cite{TVVY}.
\end{remark*}

We will prove Theorem \ref{thm:broadcasting} via a general method
that relates any reconstruction algorithm (or function)\linji{added ``or function'' as Juan asks: should we use "reconstruction function" instead? we don't need it to be an algorithm.}
with the conductance of the
Glauber dynamics.
A \emph{reconstruction algorithm} is a function $A:\Omega(L)\rightarrow\{0,1\}$ (ideally efficiently computable) such
that $A(\sigma_h)$ and $\sigma(r)$ are positively
correlated.
%Note that, for each $\sigma$ and $i\ge 0$, we use $\sigma_i$ to denote the restriction of $\sigma$ to the vertices at height $i$,
%%and for any $U\subseteq V$, we use $\sigma(U)$ to denote the restriction of $\sigma$ to the set $U$,
%e.g., $\sigma(r) = \sigma_0$.
Basically, the algorithm $A$ takes the configurations at the leaves $L$ as the input and
tries to compute the configuration at the root. When the context is clear, we write $A(\sigma)$ instead of $A(\sigma_h)$.
Under the Gibbs measure $\nu_h$, the {\em effectiveness} of $A$ is
the following measure of the covariance between the algorithm $A$'s output
and the marginal at the root of the actual measure:
\juan{Could $r_{h,A}$ be negative? If that is the case we may be in trouble in some of the proofs (I don't think that just putting an absolute value will work, but because we only look at the "effective case" we might be able to assume $r_{h,A} >0$ for all $h$(or at least $h>h_0$, i.e. looks like just a technicality) )}
\linji{To Juan: We discussed about it, and we think it is fine to define the effectiveness in this limit way, anyway it is talking about the asymptotic case, i.e., $h$ and $n$ are going to infinity. Actually, any good algorithm should have positive $r$ since a trivial algorithm would be fixing the output to $1$ always.}
\[
r_{h,A}=\min_{x\in\{0,1\}}\left[
\nu_h(A(\sigma)=\sigma(r)=x)  -\nu_h(A(\sigma)=x)  \nu_h(\sigma(r)=x)\right].
\]
If it is the case that $\liminf_{h\rightarrow\infty}r_{h,A} = c_0 >0$ for some positive constant $c_0$ depending only on $\omega$ and $b$, then
we say that it is an {\em effective reconstruction algorithm}. \linji{changed to the liminf and added a constant}
In words, an effective algorithm, is able to recover the spin at the root,
from the information at the leaves, with a nontrivial success, when $h\rightarrow \infty$.
Notice that reconstruction (defined in \eqref{eq:reconstruction-defn}) is a
necessary condition for any reconstruction algorithm to be effective, since
\juan{Aren't we missing a probability term in the middle term (for the first inequality). And for me the second inequality is not clear}
\linji{No, it is the same as the definition of the reconstruction; changed the equations a little bit.}
\[
\ExpSub{\sigma\sim\nu_h}{\left| \muhbd{h}{\sigma_h}(r\in\tau) - \frac{\omega}{1+\omega}\right|}
\ge \ExpSub{\sigma\sim\nu_h}{\big( \muhbd{h}{\sigma_h}\left(r\in\tau\right) - \nu_h\left(r\in \sigma\right)\big)\indicator{A\left(\sigma\right)=1}}
\ge r_{h,A},
\]
where $\indicator{}$ is the indicator function.
We define the \emph{sensitivity} of $A$, for the configuration $\sigma\in \Omega(T_{h})$, as the
fraction of vertices~$v$ such that switching the spin at~$v$ in $\sigma$ changes the final result of $A$.
More precisely, let $\sigma^{v}$ be the
configuration obtained from changing~$\sigma$ at~$v$. Define the sensitivity as:
\[
S_A(\sigma)  = \frac{1}{n}\#\{v\in L: A(\sigma^v)\neq A(\sigma)\}.
\]
The {\em average sensitivity} (with respect to the root being occupied)
$\bar{S}_A$ is hence defined as
\[
\bar{S}_A = \ExpSub{\sigma\sim\nu_h}{S_A(\sigma)\indicator{A(\sigma) = 1}}.
\]
It is fine to define the average sensitivity without the indicator function,
which only affects a constant factor in the analysis. We are doing so
to simplify some of the results' statements and proofs.

Typically when one proves reconstruction, it is done by presenting
an effective reconstruction algorithm.
Using the following theorem, by further analyzing the sensitivity
of the reconstruction algorithm, one obtains a lower bound on the
relaxation time or mixing time of the Glauber dynamics.
\begin{theorem}
\label{thm:sensitivity}
Suppose that $A$ is an effective reconstruction algorithm.
Then, the relaxation time $\Trel$ of the
Glauber dynamics satisfies $\Trel = \Omega\left((\bar{S}_A)^{-1}\right)$.
\end{theorem}

\begin{remark*}
The above theorem can be generalized to any spin system.
To illustrate the usefulness of this theorem, we note that the
lower bound on the mixing time of the Glauber dynamics for $k$-colorings
in the reconstruction region proved in \cite{TVVY} fits this
conceptually appealing framework.
\end{remark*}

\begin{proof}
Throughout the proof let $\nu := \nu_h$.
Consider the set $U=\left\{  \sigma:A(\sigma)=1\right\}$.
Then,
\begin{equation*}
\Phi_{U} =\frac{\textstyle\sum\nolimits_{\sigma \in U}\nu(\sigma)\sum\nolimits_{w\in L}\sum\nolimits_{\tau:\tau(w)\neq \sigma(w)}P(\sigma,\tau)}{\nu(U)(1-\nu(U))} \le
\frac{\sum_{\sigma\in U}\nu(\sigma) S_A(\sigma)}{\nu(U)(1-\nu(U))}.
\end{equation*}
From the definition of $r_{h,A}$, we have that $\nu(U) \geq \nu(A(\sigma) = \sigma(r) = 1) \geq r_{h,A}$, and similarly
$(1-\nu(U))\geq \nu(A(\sigma) = \sigma(r) = 0) \geq r_{h,A}$.
Now, because the algorithm is effective, we have $\liminf_{h\rightarrow\infty}(r_{h,A})= c_0 >0$ and hence for all $h$ big enough, $r_{h,A}>0$.
Therefore, $\Phi_{U} \leq (r_{h,A})^{-2} \bar{S}_A$, which
concludes that
$
\Trel= \gap^{-1} \ge 1/\Phi_{U} =\Omega((\bar{S}_A)^{-1})
$. \linji{changed the proof}
\end{proof}
To prove Theorem \ref{thm:broadcasting},
we analyze the sensitivity of the reconstruction algorithm
by Brightwell and Winkler \cite[Section 5]{BW} which yields the best known upper bounds on the reconstruction threshold. Our goal is to show that the
average sensitivity of this algorithm is small. The analysis of the sensitivity of the
Brightwell-Winkler (BW) algorithm, which then proves Theorem \ref{thm:broadcasting},
is presented in Section \ref{sec:lower-bound-broadcasting}.\linji{Eric added a few words in this paragraph.}

Our main objective remains of constructing a sequence of ``bad'' boundary conditions under which the Glauber dynamics
for the hard-core model
slows down in the reconstruction region. An initial approach is that if we can find a complete tree~$T'$ with some boundary
condition such that the marginal of the root being occupied exactly
equals $\omega/(1+\omega)$, then by attaching the same tree~$T'$ with the corresponding
boundary conditions to all of the leaves of a complete tree~$T$, we are able to simulate the nonuniform hard-core model on~$T$,
(i.e., the resulting measure projected onto $T$ is the same as the one in the broadcasting
model) and hence we can do the same approach to
upper-bound the conductance of the dynamics on this new tree.
However, from a cardinality argument, not for every $\omega$ there exists a complete tree of finite height
with some boundary condition such that the marginal probability of the root being occupied equals $\omega / (1+\omega)$.
Alternatively, we give a constructive way to find boundary conditions that approximate the desired marginal probability relatively accurately.
This is done in Section \ref{sec:lower-bound-hardcore}.

Finally, at the end of Section \ref{sec:lower-bound-hardcore}
we argue that since the error is shrinking very fast from the bottom level
under our construction of boundary conditions, we can again analyze the sensitivity of the
Brightwell-Winkler algorithm starting from just a few levels above the leaves.
This approach yields the lower bound stated in Part \ref{thm:above-lower}
of Theorem \ref{thm:main}.

\section{Lower Bound for Broadcasting: Proof of Theorem~\ref{thm:broadcasting}}
\label{sec:lower-bound-broadcasting}

Throughout this section we work on the broadcasting model.
To prove Theorem \ref{thm:broadcasting} we
analyze the average sensitivity of the
reconstruction algorithm used by Brightwell and Winkler \cite{BW},
which we refer to as the BW algorithm. For any configuration $\sigma$ as the input,
the algorithm works in a bottom up manner labeling each vertex from the leaves:
a parent is labeled to occupied if all of its children are labeled to unoccupied; otherwise, it is labeled to
unoccupied. The algorithm will output the labeling of the root as the final result. Formally,
it can be described by the following deterministic
recursion deciding the labeling of every vertex:
\[
R_{\sigma}(v)  =\left\{
\begin{array}
[c]{cc}%
\sigma(v)  \text{ } & \text{if }\operatorname* v\in L\\
1 - \max\{R_{\sigma}(w_1),R_{\sigma}(w_2),\ldots,R_{\sigma}(w_b)\}  \text{ } & \text{otherwise}
\end{array}
\right.
\]
where $w_{1},\ldots,w_{b}$  are the children of $v$.
Finally, let $\A(\sigma) =\A(\sigma_h) =R_{\sigma}(r)$.
Note that, $\A(\sigma)$ only depends on the configuration $\sigma_h$ on the leaves.
The algorithm is proved to be effective in \cite{BW} when $\delta>0$.
Therefore, it can be used in our case to lower bound the relaxation time.
In this algorithm, by definition we have
\begin{equation}
\label{eq:s31}
\bar{S}_\A = O\big(n^{-1}\ExpSub{\sigma\sim\nu_h}{\#\{v\in L: \A(\sigma) = 1 \textrm{ and } \A(\sigma^{v}) = 0\}}\big),
\end{equation}
Due to the symmetry of the function $R_{\sigma}(v)$ and the measure $\nu_h$, the expectation can be further simplified as
\begin{equation}
\label{eq:s32}
\ExpSub{\sigma}{\#\{v\in L: \A(\sigma) = 1 \textrm{ and } \A(\sigma^{v}) = 0\}}=b^h\nu_h(\A(\sigma) = 1 \textrm{ and } \A(\sigma^{\hv}) = 0),
\end{equation}
where $\hv$ is now a fixed leaf.
To bound the right hand side of Eq.\eqref{eq:s32}, let $\kappa\in\Omega(T_h)$ be a fixed configuration such that $\A(\kappa) = 1$.
Let the path $\mathcal{P}$ from $\hv$ to the root $r$ be $u_{0} = \hv \leadsto u_{1}\leadsto\cdots\leadsto u_{h} = r$,
and for any $i > 0$, let $w_{i,j}$ be the children of $u_i$
so that the labeling is such that for $j =1$, $w_{i,1} = u_{i-1}$
and for $j\neq 1$, $w_{i,j}$ is not on the path~$\Pa$.
An important observation is that, in order to make $\A(\kappa)$ change to $0$
by changing only the configuration at $\hv$ of $\kappa$, a necessary condition for $\kappa$ is $R_\kappa(u_{i}) = 1 - R_\kappa(u_{i-1})$ for all $i \ge 1$.
Then for all $i\ge 1$ and $j\in\{2,\ldots,b\}$, we have $R_{\kappa}(w_{i,j}) = 0$.
To calculate the probability that a random $\kappa\sim\nu_h$ satisfies such conditions, it would be
easier if we expose the configurations along the path $\Pa$. Since then, conditioning
on the configurations on the path, the events $R_{\kappa}(w_{i,j}) = 0$ are independent for all $i,j$.
And if $\kappa(u_i) = 0$, we have for all $j>1$, the conditional probability of $R_{\kappa}(w_{i,j}) = 0$
equals $\ProbSub{\eta\sim\nu_{i-1}}{\A(\eta) = 0}$, the probability $\A$ algorithm outputs a $0$ over a random configuration $\eta$ of the leaves of
the complete tree $T_{i-1}$ with height $i-1$.
The analysis above leads to the following lemma, which bounds the probability $\nu_h(\A(\sigma) = 1 \textrm{ and } \A(\sigma^{\hv}) = 0)$.
\linji{changed $v$s to $u$s, added explanation of the key lemma here.}
\begin{lemma}
\label{lem:nonuni1}
For every $i > 0$, let $\eta \in \Omega(T_{i-1})$ be a configuration chosen randomly according to measure $\nu_{i-1}$, then
\[
\nu_{h}(\A(\sigma) = 1 \textrm{ and } \A(\sigma^{\hv}) = 0) \le \ExpSub{\kappa\sim\nu_h}{\prod\limits_{i>0:\kappa(u_i)=0} \ProbSub{\eta\sim\nu_{i-1}}{\A(\eta) = 0}^{(b-1)}}.
\]
\end{lemma}

Complete proofs of lemmas in this section are deferred to Section \ref{App:broadcasting}. To use Lemma~\ref{lem:nonuni1},
we derive the following uniform bound on the probability $\ProbSub{\eta\sim\nu_{i}}{\A(\eta) = 0}$,
for all~$i$. Here and through out the paper, $b_0(\delta)$ is a function explicitly defined in Lemma~\ref{lem:b0h} in Section~\ref{App:tlemmas}.
This function is of order $\exp(\delta^{-1}\ln(\delta^{-1}))$ as $\delta\to 0$ and remains bounded as $\delta \to \infty$.\ricardo{Info about $b_0(\delta)$ introduced, and the statement of the following lemma was changed.}

\begin{lemma}
\label{lem:nonuni2}
Let $\delta>0$, and let $\omega=(1+\delta)\ln b/{b}$. For all
$b\geq \BW$ and $i \geq 1$,
\[
\ProbSub{\eta\sim\nu_i}{\A(\eta) = 0}\leq \frac{(1.01)^{1/b}}{1+\omega}.
\]
\end{lemma}

Combining Equations \eqref{eq:s31}, \eqref{eq:s32}, Lemma \ref{lem:nonuni1} and Lemma \ref{lem:nonuni2}, we are able to upper bound the
average sensitivity of the BW\ algorithm:
\begin{equation*}
\bar{S}_\A = O\big( \nu_h(\A(\sigma) = 1 \textrm{ and } \A(\sigma^{\hv}) = 0) \big)
= O\left(\ExpSub{\kappa\sim\nu_h}{  \left(  \frac{1.01\omega(1+\omega)}{\lambda}\right)^{\#\left\{  i:\kappa(u_{i})=0\right\}}}\right)  \text{.}
\end{equation*}

In this expectation, the number of unoccupied vertices in the path $\mathcal{P}$ can be trivially lower bounded by $h/2$, since it is impossible that there exists $i > 0$, $\kappa(u_i) = \kappa(u_{i-1}) = 1$.
Therefore, the above expectation can be easily bounded by $O^*(n^{-(1+\delta)/2})$. This is not good enough in our case. We sharpen the bound using Lemma \ref{lem:asympunn}
in Section \ref{App:tlemmas}, leading to the following theorem, whose complete proof is contained in Section \ref{App:broadcasting}.
\begin{theorem}
\label{th:slownonunif}
Let $\delta > 0$, and let $\omega=(1+\delta)\ln b /{b}$. For all
$b\geq \BW$,
\juan{Actually $\ln\left( \lambda/(1.01\omega b)^2\right) = b\ln\left((1+\omega)/1.01\omega b^2\right)$ could we use this one?}
\linji{Is it slightly wrong? the power $b$ is not over the whole thing.}
\[
\Trel = \Omega\big(n^{d}\big)\text{,\quad where }d=\left(  1+\frac{\ln\left(
\lambda/(1.01\omega b)^2\right)}{2\ln b}\right).
\]
\end{theorem}
Theorem \ref{thm:broadcasting} is a simple corollary of Theorem \ref{th:slownonunif} by noticing that $d=1+\delta/2-O\left(\frac{\ln\ln b}{\ln b}\right)$. Furthermore, we can ``hide'' the fact that $b\geq b_0(\delta)$ in this residual term, by using the trivial lower bound $\Omega(n)$ for all $b<b_0(\delta)$ and $b_0(\delta)\approx \exp(\delta^{-1}\ln(\delta^{-1}))$ as $\delta \to 0$.

\subsection{Proofs.}
\label{App:broadcasting}
Note that throughout this paper, we will use the following notations for the relationships between two functions $f(x)$ and $g(x)$ for simplicity. If $\lim_{x\rightarrow \infty} f(x)/g(x) = 1$, we write $f(x) \approx g(x)$; if $f(x) = O(g(x))$, we write $f(x) \lesssim g(x)$ and if $f(x) = \Omega(g(x))$, we write $f(x) \gtrsim g(x)$. \linji{added this paragraph.}

\begin{proof}[Proof of Lemma \ref{lem:nonuni1}]
Let $\mathbf{x} = \{0,1\}^h$ be a valid configuration on the path $\mathcal{P}$.
Conditioning on $\kappa(u_i) = \mathbf{x}(i)$ for all $i$, we know that the events $R_{\kappa}(w_{i,j}) = 0$ are independent for all $i$ and $j$. Given $\kappa(u_i) = 0$, the
probability of the event $R_{\kappa}(w_{i,j}) = 0$ equals $\ProbSub{\eta\sim\nu_{i-1}}{\A(\eta) = 0}$; and given $\kappa(u_i) = 1$, the probability of this event can be trivially upper bounded by $1$ (this bound is ``safe'', in the sense that the actual quantity is close to $1$ for big $\lambda$).
By this, we can conclude that
\begin{align*}
\nu_{h}(\A(\kappa) & =1\text{ and }\A(\kappa^{\hv})=0)\\
& \leq \sum\limits_{\mathbf{x}}\nu_{h}(\kappa:\forall i,\kappa
(u_{i})=\mathbf{x}(u_{i}))\prod\limits_{i>0:\kappa(u_{i})=0}%
\ProbSub{\eta\sim\nu_{i-1}}{\A(\eta) = 0}^{(b-1)}\\
&
=\ExpSub{\kappa}{\prod\limits_{i>0:\kappa(u_i)=0} \ProbSub{\eta\sim\nu_{i-1}}{\A(\eta) = 0}^{(b-1)}}.
\end{align*}
\end{proof}

\begin{proof}[Proof of Lemma \ref{lem:nonuni2}]
In the proof, we will use the fact that $\exp\left(\frac{2(1.01)(\omega b)^2 }{\lambda}\right)\leq 1.01$, whenever $b\geq \BW$ (Lemma \ref{lem:b0h}). Now, for simplicity, denote $\tg_i =  \ProbSub{\eta\sim\nu_{i-1}}{A(\eta) = 0}$.
First of all, notice the recurrences%
\begin{align*}
\tg_{i+1} &  =\frac{\omega}{1+\omega}\left(
1-\left(  1-  \tg_{i-1}^b\right)
^b \right)
+\frac{1}{1+\omega}\left(1-\tg_{i}^b\right)  \text{,}\\
\tg_{1} &  =\frac{1}{1+\omega}%
\quad\text{,\quad}
\tg_{2}=
\frac{1}{1+\omega}\left(1-\left(\frac{1}{1+\omega}\right)^{b}\right)  \text{.}%
\end{align*}
The result follows by an easy induction: For $h=1,2$, the result is
clear. On the other hand, from
the previous recurrences, if it is the case that $\tg_{i}\leq \frac{(1.01)^{1/b}}{1+\omega}$, then%
\begin{align*}
\tg_{i+1}  ^{b}  &
\leq\left[  \frac{\omega}{1+\omega}\left(  1-\left(  1-
\tg_{i-1}  ^{b}\right)  ^{b}\right)  +\frac{1}{1+\omega}\right]  ^{b}\\
& \leq\left[  \frac{\omega}{1+\omega}\left(  1-\left(  1-\frac{1.01\omega  }{\lambda}\right)  ^{b}\right)  +\frac{1}{1+\omega
}\right]  ^{b}\\
& \leq\left(  \frac{1+\frac{1.01\omega^2b  }%
{\lambda}}{1+\omega}\right)  ^{b}\leq\frac{\exp(1.01(\omega b)^2/\lambda)}{(1+\omega)^b}\leq\frac{1.01}{(1+\omega)^b}\text{,}%
\end{align*}
where the third inequality follows from the fact that $(1-u)^{b}\geq1-ub$ for $u<1$,
the fourth inequality follows from $(1+u)\leq e^{u}$, and the last inequality follows from the fact
that  $\exp\left(\frac{2(1.01)(\omega b)^2 }{\lambda}\right)\leq 1.01$ for $b\geq \BW$.
\end{proof}

\begin{proof}[Proof of Theorem \ref{th:slownonunif}]
From Lemma \ref{lem:asympunn} in Section~\ref{App:tlemmas}, we have that
\[
\Exp{  \left(  \frac{1.01\omega(1+\omega)  }{\lambda}\right)
  ^{\#\{  i:\sigma(u_{i})=0\}}}
\approx
\left(  1+\frac{  1-\epsilon  }{2\epsilon (1+\omega)}\right)
\frac{\left(  1+\epsilon\right)  }{2}
\left(
\frac{1.01 \omega}{2\lambda}
\left[
1+\sqrt{1+4\lambda/1.01}
\right]
\right)^{h}%
\]
where $\eps =\left[\sqrt{1+4\lambda/1.01}\right]^{-1}$. The previous term is asymptotically dominated by $\left(1.01\frac{\omega}{\lambda^{1/2}}\right)^h$. Therefore,
\[
\bar{S}_A =O \left(\left[\frac{1.01\omega}{\lambda^{1/2}}\right]^h\right)
=O\left(n^{-\left[
1+\frac{\ln\left(\lambda/(1.01\omega b)^2\right)}{2\ln b}
\right]}\right)\text{.}
\]

%The last inequality is a straightforward calculation by plugging in $h = \log_b n$.

Now, from \cite[Section 5]{BW}, it is known that the BW\ algorithm is effective for
$\omega> (1+\delta)\ln b/{b}$ and $b>\BW$. (This can also be deduced from
Lemma \ref{lem:nonuni2} by following the same steps as in Proposition \ref{prop:eff2}), therefore Theorem \ref{thm:sensitivity} applies. The conclusion follows.
\end{proof}

\subsection{Some Technical Lemmas.}
\label{App:tlemmas}

\begin{lemma}
\label{lem:b0h}
Define $b_0(\delta)=\min\{ b_0:\exp\left( \frac{2(1.01)(\omega b)^2 }{\lambda}\right)\leq 1.01 \text{ for all } b\geq b_0 \}$. Then $b_0(\delta)$ is a continuous function such that
\begin{enumerate}
\item $b_0(\delta)<\infty$ for all $\delta>0$ (that is, it is well defined).
\item $b_0(\delta)\approx\exp\left((1+o(1))\delta^{-1}\ln(\delta^{-1})\right)$ as $\delta \to 0$.
\item $b_0(\delta)\approx b_0(\infty)$ as $\delta \to \infty$, where $ b_0(\infty)$ is a fixed constant $\leq 2$.
\end{enumerate}
\end{lemma}
\begin{proof}
For (1), just notice that for $\delta>0$ fixed, we have that, as $b\to \infty$, \[\exp\left( \frac{2(1.01)(\omega b)^2 }{\lambda}\right) \approx 1<1.01.\]

For (2), notice that if we let $b=\exp\left(\beta\delta^{-1}\ln(\delta^{-1})\right)$, then as $\delta\to 0$, we have that \[\exp\left( \frac{2(1.01)(\omega b)^2 }{\lambda}\right) \approx \exp\left(2(1.01)\beta\delta^{\beta-1}\ln(\delta^{-1})\right).\] Therefore, if $\beta\leq 1$, it is the case that  $\exp\left( \frac{2(1.01)(\omega b)^2 }{\lambda}\right) \to \infty$ as $\delta \to 0$, while if $\beta>1$ the same expression goes to $1$.

For (3), notice that for $b$ fixed,  $\exp\left( \frac{2(1.01)(\omega b)^2 }{\lambda}\right) \approx \exp\left( \frac{\Theta(1) }{\delta^{b-1}}\right)$, as $\delta\to \infty$.
\end{proof}

\begin{lemma}
\label{lem:asympunn}Let $\zeta_{0},\zeta_{1},\ldots$ be a Markov process with
state space $\left\{  0,1\right\}  $, such that $\zeta_{0}=0$ and with
transition rates $p_{0\rightarrow0}=p$, $p_{0\rightarrow1}=q$,
$p_{1\rightarrow0}=1$, $p_{1\rightarrow1}=0$. Let $N_{h}=\#\left\{  1\leq
i\leq h:\zeta_{i}=0\right\}  $, then
\[
\Exp{  a^{N_{h}}}  \approx\left(  1+\frac{p\left(
1-\epsilon\right)  }{2\epsilon}\right)  \frac{\left(  1+\epsilon\right)  }%
{2}\left(  \frac{pa}{2}\left[  1+\sqrt{1+4q/\left(  ap^{2}\right)  }\right]
\right)  ^{h}\text{.}
\]
where $\epsilon=\frac{1}{\sqrt{1+4q/\left(  ap^{2}\right)  }}$.
Moreover, if the transition rate $p_{0\rightarrow0}$ is inhomogeneous but such
that $\left\vert p-p_{0\rightarrow0}^{\left(  i\right)  }\right\vert
\leq\delta$, then
\[
\Exp{  a^{N_{h}}}
\lesssim
\left(1+\frac{\left(p+\delta\right)\left(  1-\bar{\epsilon}\right)}{2\bar{\epsilon}}\right)
\frac{\left(  1+\bar{\epsilon}\right)}{2}
\left(\frac{\left( p+\delta\right)a}{2}
\left[1+\sqrt{1+4\left(q+\delta\right)/\left(a\left(p+\delta\right)^{2}\right)  }
\right]
\right)^{h}
\text{,}%
\]
where $\bar{\epsilon}=\frac{1}{\sqrt{1+4\left(  q+\delta\right)  /\left(
a\left(  p+\delta\right)  ^{2}\right)  }}$.
\end{lemma}
\begin{proof}
Straightforward combinatorics leads to the expression
\[
\Exp{  a^{N_{h}}}  =
{\textstyle\sum\nolimits_{k=0}^{\left\lfloor h/2\right\rfloor }}
\tbinom{h-k}{k}p^{h-2k}q^{k}a^{h-k}+
{\textstyle\sum\nolimits_{k=1}^{\left\lfloor (h+1)/2\right\rfloor }}
\tbinom{h-k}{k-1}p^{h-2k+1}q^{k}a^{h-k}
\]
Now, for the first term, we have that
\[{\textstyle\sum\nolimits_{k=0}^{\lfloor h/2\rfloor}}
\tbinom{h-k}{k}p^{h-2k}q^{k}a^{h-k}
= (pa)^{h}
{\textstyle\sum\nolimits_{k=0}^{ \lfloor h/2 \rfloor }}
\tbinom{h-k}{k}x^{k},\]
where $x=\frac{q}{ap^{2}}$, which by standard saddle point methods, noticing that the function
$\displaystyle\phi(t)=\lim_{h \to \infty} h^{-1} \ln \left[\tbinom{h-th}{th}x^{th}\right]$
reaches its maximum at the point $t^{\ast}=\frac{1}{2}(1-\eps)$, where $\eps=1/{\sqrt{1+4x}}$ and $\phi^{\prime\prime}(t^{\ast})
=\sqrt{\frac{4}{\eps(1-\eps)(1+\eps)}}$, implying that
\[
\frac{1}{(pa)^{h}}
{\textstyle\sum\nolimits_{k=0}^{ \lfloor h/2 \rfloor }}
\tbinom{h-k}{k}p^{h-2k}q^{k}a^{h-k}
\approx
\frac{(1+\eps)}{2}\left(\frac{1+\sqrt{1+4x}}{2}\right)^{h}.
\]
Similarly,
\[
\frac{1}{(pa)^{h}}
{\textstyle\sum\nolimits_{k=1}^{ \lfloor (h+1)/2 \rfloor }}
\tbinom{h-k}{k-1}p^{h-2k+1}q^{k}a^{h-k}
\approx
\frac{p(1-\eps)}{2\eps}
\frac{(1+\eps)}{2}\left(\frac{1+\sqrt{1+4q/(ap^{2})}}{2}\right)^{h},
\]
from where the result follows.
In the inhomogeneous case, we have that
\[
\Exp{  a^{N_{h}}}  \leq%
{\textstyle\sum\nolimits_{k=0}^{\left\lfloor h/2\right\rfloor }}
\tbinom{h-k}{k}\left(  p+\delta\right)  ^{h-2k}\left(  q+\delta\right)
^{k}a^{h-k}+%
{\textstyle\sum\nolimits_{k=1}^{\left\lfloor (h+1)/2\right\rfloor }}
\tbinom{h-k}{k-1}\left(  p+\delta\right)  ^{h-2k+1}\left(  q+\delta\right)
^{k}a^{h-k}\text{,}%
\]
from where the result follows using the same asymptotic.
\end{proof}

\section{ ``Bad" Boundary Conditions: Proof of Theorem~\ref{thm:main}.\ref{thm:above-lower}}
\label{sec:lower-bound-hardcore}

%I move the paragraph here to the end of the tex file.
First, we will show that for any $\omega$,
there exists a sequence of boundary conditions,
denoted as $\GammaN := \{\BD_i\}_{i>0}$,
one for each complete tree of height $i>0$,
such that if $i\rightarrow \infty$, the probability of the root being occupied converges to $\frac{\omega}{1+\omega}$.
Later in this section we will exploit such a construction to attain in full
the conclusion of Part \ref{thm:above-lower} of Theorem~\ref{thm:main}.

As a first observation, note that, the Gibbs measure for the hard-core model on
$T_i$  with boundary condition $\BD$ is the same as the
Gibbs measure for the hard-core model (with the same activity $\lambda$)
on the tree $T$ obtained from $T_i$ by deleting all of the leaves as well as the parent of each (occupied)
leaf $v\in\BD$.
It will be convenient to work directly with such ``trimmed'' trees, rather than the
complete tree with boundary condition.
Having this in mind, our construction will be inductive in the following way.
We will define a sequence of (trimmed) trees $\{(L_i,U_i)\}_{i\geq 0}$ such
that $L_{i+1}$ is comprised of $s_{i+1}$ copies of $U_i$ and $b-s_{i+1}$ copies
of $L_i$ with $\{s_i\}_{i \geq 1}$ properly chosen.
Similarly, $U_{i+1}$ is comprised of $t_{i+1}$ copies of $U_i$ and $b-t_{i+1}$
copies of $U_i$, with $\{t_i\}_{i \geq 1}$ properly chosen.

We will show that, for either $T^*_i=L_i$, or $T^*_i=U_i$,
it is the case that the \textquoteleft$Q$\textquoteright-value, defined as:\linji{added a sentence after ``where...''.}
\[
Q(T^*_i) = \frac{ \mu_{T^*_i}\left(\sigma(r)=1\right)}{\omega \mu_{T^*_i}\left( \sigma(r)=0\right)},
\]
where $\mu_{T^*_i}(\cdot)$ is the hard-core measure on the trimmed tree $T^*_i$, satisfies $Q(T^*_i) \rightarrow 1$.
Note that if $Q(T^*_i) = 1$, then the probability of the root being occupied is $\omega/(1+\omega)$ as desired.
To attain this, we will
construct $L_i$ and $U_i$ in such a way that $Q(U_i) \ge 1$ and $Q(L_i) \le 1$.

%Generally, we can define the so-called $Q$ value for any tree $T$: $Q(T)$ is defined as the ratio between the %probability of the root being occupied and that of the root being unoccupied times a factor of $\omega^{-1}$.
The recursion for $Q(L_{i+1})$ can be derived easily as
\[
Q(L_{i+1}) =  \frac{(1+\omega)^b}{(1+\omega Q(U_i))^{s_{i+1}}(a+\omega Q(L_i))^{b-s_{i+1}}},
\]
and a similar equation holds for $Q(U_{i+1})$ by replacing $s_{i+1}$ with $t_{i+1}$.

%Q(U_{i+1}) =  \frac{(1+\omega)^b}{(1+\omega Q(U_i))^{t_{i+1}}(a+\omega Q(L_i))^{b-t_{i+1}}}\text{ .}

To keep the construction simple,
we inductively define the appropriate $t_i$ and $s_i$, so that once $L_i$ and $U_i$ are given, we let $t_{i+1}$ be the minimum choice so that the resulting $Q$-value
is $\geq 1$,
more precisely, we let:
\[
t_{i+1}=\arg\min_{\ell}\{Q = \frac{(1+\omega)^b}{(1+Q(U_{i}))^\ell(1+\omega Q(L_{i}))^{b-\ell}}: Q \ge 1\}.
\]
\[
\hbox{And similarly, we let:}~~s_{i+1}=\arg\max_{\ell}\{Q = \frac{(1+\omega)^b}{(1+Q(U_{i}))^\ell(1+Q(L_{i}))^{b-\ell}}: Q \le 1\}.
\]

The recursion starts with $U_1$ being the graph of a single node and $L_1$ being the empty set, so that $Q(U_1) = \lambda/\omega$ and $Q(L_1) = 0$.
Observe that, by definition, $s_{i+1}~\in~\{t_{i+1},t_{i+1}+1\}$ and that the construction guarantees that
the values $Q(L_i)$ are at most $1$, and the values $Q(U_i)$ are at least $1$. The following simple lemma justifies the correctness of our construction.
\begin{lemma}
\label{lem:err1}
\[
\lim_{i\rightarrow \infty} {Q(U_i)}/{Q(L_i)} = 1.
\]
\end{lemma}
\begin{proof}
It is easy to see that either $t_i=s_i$ (meaning that $Q(L_i)=Q(U_i)=1$), or $t_i=s_i-1$, which implies that
\[
\frac{Q(U_i)}{Q(L_i)} = \frac{1+\omega Q(U_{i-1})}{1+\omega Q(L_{i-1})} < \frac{Q(U_{i-1})}{Q(L_{i-1})}.
\]
So the ratio is shrinking. Suppose the limit is not $1$ but some value $q>1$. Then,
\[
\frac{Q(U_{i-1})}{Q(L_{i-1})}-\frac{Q(U_{i})}{Q(L_{i})} = \frac{Q(U_{i-1})-Q(L_{i-1})}{(1+\omega Q(L_{i-1}))Q(L_{i-1})}.
\]
Since $Q(U_i)/Q(L_i) > q$ and $Q(L_i)\le 1$, we have
\[
\frac{Q(U_{i-1})-Q(L_{i-1})}{(1+\omega Q(L_{i-1}))Q(L_{i-1})}\ge
\frac{(q-1)Q(L_{i-1})}{Q(L_{i-1})(1+\omega)}=\frac{q-1}{1+\omega},
\]
which is a constant.

Therefore as long as $q>1$, we show that the difference between the ratios for each step $i$ is at least some constant which is impossible. Hence the assumption is false.
\end{proof}

By this lemma, it is easy to check that if we let $T^*_i$ to be equal to either
$U_i$ or $L_i$,  then $Q(T^*_i)\rightarrow 1$.
Indeed, we can show that the additive error decreases exponentially fast. The following lemma indicates that this is the case for $\omega<1$ (although a similar result holds for any $\omega$).
\begin{lemma}
\label{lem:err2}
Let $\eps^+_i$ be the value of $Q(U_i)-1$ and let $\eps^-_i$ be the value of $1-Q(L_i)$, then
\[
\eps^+_{i+1}+\eps^-_{i+1} \le \omega(\eps^+_i + \eps^-_i).
\]
\end{lemma}
\begin{proof}
We can rewrite the expression
\[
(1+\omega)^b / (1+\omega Q(U_i))^j (1+\omega Q(L_i))^{b-j}
\]
as
\[
\frac{1}{(1+\frac{\omega}{1+\omega}\eps^+_i)^j(1-\frac{\omega}{1+\omega}\eps^-_i)^{b-j}}.
\]
Now, let $k$ be the biggest index over $[b]$ such that the denominator of the previous expression is less than $1$ (thus, $k+1$ will be the least index such that the denominator is greater than $1$). Then,

\begin{eqnarray*}
\eps^+_{i+1}+\eps^-_{i+1} & = & \frac{1}{(1+\frac{\omega}{1+\omega}\eps^+_i)^k(1-\frac{\omega}{1+\omega}\eps^-_i)^{b-k}}-
\frac{1}{(1+\frac{\omega}{1+\omega}\eps^+_i)^{k+1}(1-\frac{\omega}{1+\omega}\eps^-_i)^{b-k-1}}\\
&=&
\frac{\frac{\omega}{1+\omega}(\eps^+_i+\eps^-_i)}{(1+\frac{\omega}{1+\omega}\eps^+_i)^{k+1}(1-\frac{\omega}{1+\omega}\eps^-_i)^{b-k}}\\
&\le&
\frac{\frac{\omega}{1+\omega}(\eps^+_i+\eps^-_i)}{1-\frac{\omega}{1+\omega}\eps^-_i}~~~~\textrm{(by the property of $k+1$)}\\
&\le&
\omega(\eps^+_i + \eps^-_i).
\end{eqnarray*}
\end{proof}

Coming back to the original tree-boundary notation, let $\BDU{h}$ be the boundary corresponding to the trimming of the tree $U_h$ and let $\BDL{h}$ be the boundary corresponding to the  trimming of the tree $L_h$.
By our construction, for any vertex $v$ on the tree of height $h$,
the measure from $\muhbd{h}{\BDU{h}}$ (or $\muhbd{h}{\BDL{h}}$) projected onto the space of the independent sets of the subtree rooted at $v$ with the boundary condition corresponding
to the correct part of $\BD$ and the parent of $v$ being unoccupied is either $\muhbd{i}{\BDU{i}}$ or $\muhbd{i}{\BDL{i}}$, where $i$ is the distance of $v$ away from the leaves on $T_h$.
Conditioning on the parent of $v$ being unoccupied, in the broadcast process
defined in the Introduction, we would occupy $v$ with probability $\omega/(1+\omega)$.
Therefore, in the above construction, the probability $v$ is occupied (or
rather unoccupied) is close to the desired probability, and the error will
decay exponentially fast with the distance from the leaves.
This is formally stated in the following corollary of Lemma \ref{lem:err2}.

\begin{corollary}
\label{cor:err1}
Given any $\omega<1$ and the complete tree of height $i$, for $\BD$ equal to~$\BDU{i}$ or~$\BDL{i}$ inductively constructed above, we have
\[
\left\vert \muhbd{i}{\BD}(\sigma(r)=0)-\frac{1}{1+\omega}\right\vert \leq \omega^{i-1}\lambda/b.
\]
\end{corollary}

%This next lemma says that for vertices at least a sufficiently large
%constant distance from the leaves, their marginal distribution
%(conditioned on their parent being unoccupied) in the above construction
%is good enough for a further induction to work.

%\begin{corollary}
%\label{cor:err2}
%There is a constant $c(\lambda,b)$ such that if $i\geq c(\lambda,b)$, then
%\[
% \muhbd{i}{\BD}(\sigma(r)=0)\leq  (1+\omega)^{-1}{\exp\left(\frac{\omega \ln(\lambda b)}{\lambda}\right)}\text{.}
%\]
%\end{corollary}

%The above corollary follows from Lemma \ref{lem:err2} after observing
%that there is a constant $c(\lambda,b)$ (actually $c(\lambda,b)=\BH$ suffices), such that for $i\geq c(\lambda,b)$,
%\begin{equation}
%\label{eq:analytic222}
% \lambda \left( \frac{\omega}{1+\omega} \right)^{i-1} \leq \frac{\omega \ln(\lambda b)}{\lambda}.
%\end{equation}
%\linji{changed the statement.}

Throughout the rest of this section it is assumed that we are dealing with the boundary conditions $\{\BDU{h}\}_{h\in \mathbf{N}}$ and $\{\BDL{h}\}_{h\in \mathbf{N}}$ constructed above.
We will then show that for every $\omega = (1+\delta)\ln b/b$ under these two boundary conditions, the Glauber dynamics on the hard-core model slows down, whenever $\delta > 0$.
As we know from Corollary \ref{cor:err1}, the error of the marginal goes down very fast, so that roughly we can think of the
marginal distribution of the configurations on the tree from the root to
the vertices a few levels
 above the leaves as being close to the broadcasting measure.
 In fact, by following the same proof outline
 as we did in Section \ref{sec:lower-bound-broadcasting},
 we are able to prove the same lower bound in the hard-core model for these boundaries.
 To do that we need a slight generalization of the reconstruction algorithm
 and extensions of the corresponding lemmas used in that section to handle
 the errors in the marginal probabilities.

To generalize the notion of a reconstruction algorithm to the case of a
boundary condition we need to add an extra parameter $\hp$ depending only on $\omega$ and $b$.  We will
essentially ignore the bottom $\hp$ levels in the analysis, and
we will use that for the top $h-\hp$ levels the marginal probabilities are close
to those on the broadcasting tree.
\juan{do we need to do all this new definitions? Can we just say that we sample the coloring at level $h-\hp$ and run the reconstruction algorithm. Thus the algorithm is now "random" (we just introduced some randomness) and then we do averaging (with respect to the new randomness) in the sensitivity}
We define a \emph{reconstruction algorithm
with a parameter $\hp$} for the tree $T_h$ with boundary condition $\BD$ as a function $A_\hp:\Omega(L_{h-\hp})\rightarrow\{0,1\}$.
The algorithm $A_\hp$ takes the configurations of the vertices at height $h-\hp$ as the input
and tries to compute the configuration at the root.
For any $\sigma \in\Omega(T_{h,\BD})$,
the sensitivity is defined as:
$
S_{\hp,A}(\sigma)  = \frac{1}{n}\#\left\{v \in L_{h-\hp}: A_\hp(\sigma^v_{h-\hp})  \neq A_\hp(\sigma_{h-\hp})  \right\}
$.
The average sensitivity of the algorithm at height $h-\hp$ with respect to the boundary $\BD$ is defined as:
$\bar{S}_{\hp,A}^{\BD} = \ExpSub{\sigma}{S_{\hp,A}(\sigma)\indicator{A_{\hp}(\sigma_{h-\hp}) = 1}}
$. And the effectiveness is defined as:
\[
r_{\hp,A}^{\BD} = \min_{x\in \{0,1\}} [\mu_{h,\BD}(A_\hp(\sigma_{h-\hp}) = x \textrm{ and } \sigma(r) = x) - \mu_{h,\BD}(A_\hp(\sigma_{h-\hp}) = x)\mu_{h,\BD}(\sigma(r) = x)].
\]
We can show the analog of Theorem \ref{thm:sensitivity2} in this setting.
\begin{theorem}
\label{thm:sensitivity2}
Suppose that $A_\hp$ is an effective reconstruction algorithm.
Then, it is the case that the spectral gap $\gap$
of the Glauber dynamics for the hard-core model on the tree of height $h$ with boundary condition $\BD$,
 satisfies $\gap = O(\bar{S}_{\hp,A}^{\BD})$, and hence the relaxation time of this Glauber dynamics satisfies $\Trel = \Omega(1/\bar{S}_{\hp,A}^{\BD})$.
\end{theorem}
To bound the average sensitivity for the boundary conditions
 $\BDU{h}$ and $\BDL{h}$ constructed above,
 we again use the same BW\ algorithm as we analyzed for the broadcasting tree.
 As in Eq.\eqref{eq:s31} and \eqref{eq:s32},
 it is again enough to bound the probability \linji{I added this new notation, otherwise it is too long and looks messy.}
 \[
 P_{\hp,\A}^\BD := \muhbd{h}{\BD_h}(\A_{\hp}(\sigma_{h-\hp})~=~1 \textrm{ and } \A_{\hp}(\sigma^{\hv}_{h-\hp})~=~0)
 \]
for a fixed vertex $\hv$ at a distance $\hp$ from the leaves, although in this case, this probability will not be the same for all $\hv$. Let the path $\Pa$ from $\hv$ to the root $r$ be $u_{0} = \hv \leadsto u_{1}\leadsto\cdots\leadsto u_{h-\hp} = r$,
and for each $i > 0$ and $j\in\{1,\dots,b\}$,
let $w_{i,j}$ be defined similarly as in Section \ref{sec:lower-bound-broadcasting}. Further, let $\BD_{i,j}$ be the boundary condition $\BD_h$ restricted to the subtree $T_{w_{i,j}}$ of $T_h$ rooted at the vertex $w_{i,j}$. These subtrees are of height $i+\ell-1$ for each $i$.
Note that, by our construction of the boundary conditions, for each fixed $i$, $\BD_{i,j} = \BDU{i+\hp-1}$ or $\BD_{i,j} = \BDL{i+\hp-1}$.
The probability  $P_{\hp,\A}^\BD$ can be calculated by the following lemma,
which is the analog of Lemma \ref{lem:nonuni1} for the broadcasting tree.
\begin{lemma}
\label{lem:uni1}
\[
\ P_{\hp,\A}^\BD \le \ExpSub{\sigma}{\prod\limits_{i>0:\sigma(u_i)=0} \prod_{j=2}^{b}
\ProbSub{\eta\sim\muhbd{i+\hp-1}{\BD_{i,j}}}{A_{\hp}(\eta) = 0}
},
\]
where the expectation is over the measure $\muhbd{h}{\BD_h}$, and for each $i,j$, $\eta$ is a random configuration on the subtree rooted at $w_{i,j}$ with the probability measure $\muhbd{i+\hp-1}{\BD_{i,j}}$.
\end{lemma}

The proofs of Theorem \ref{thm:sensitivity2} and Lemma \ref{lem:uni1} use the
same proof approach as for Theorem \ref{thm:sensitivity} and Lemma \ref{lem:nonuni1} respectively. However, to bound
$\ProbSub{\eta}{A_{\hp}(\eta) = 0}$
for every $i>0$, in spite of going along the lines of Lemma \ref{lem:nonuni2}, the proof does require extra care to deal with the errors in the marginal probabilities which were bounded in Corollary \ref{cor:err1}.
In particular, we will establish the following lemma to upper bound $\ProbSub{\eta\sim\muhbd{i+\hp-1}{\BD_{i,j}}}{ A_{\hp}(\eta) = 0}$ for each $i>0$. 
Here and throughout the text, we define $\hp(\lambda,b)$ to be the minimum $\hp$ such that for all $i\geq \hp$,
\[
\ProbSub{\eta\sim\muhbd{i}{\BDL{i}}}{\eta(r) = 0}\leq \frac{1}{1+\omega}\exp\left(\frac{1.01(\omega b)^2}{\lambda}\right).
\]
The existence of such constant $\hp(\lambda,b)$ is guaranteed by Lemma \ref{lem:err1}, also from Corollary \ref{cor:err1} we can deduce a explicit value for $\hp(\lambda,b)$, provided that $\omega <1$. 
\begin{lemma}
\label{lem:uni2}
Given any $\delta>0$, and $i \geq \hp(\lambda,b) = \hp$, then both $\ProbSub{\eta\sim\muhbd{i}{\BDU{i}}}{A_{\hp}(\eta) = 0}$ and $\ProbSub{\eta\sim\muhbd{i}{\BDL{i}}}{A_{\hp}(\eta) = 0}$ are upper bounded by $\frac{1.01^{1/b}  }{1+\omega}$ for any $b\geq \BW$.
\end{lemma}
And also it is not hard to show that the BW algorithm under this setting is effective. 
\begin{proposition}
\label{prop:eff2}
The BW reconstruction algorithm is effective to recover the configuration at the root
from the configurations at distance $\hp(\lambda,b)$ from the leaves.
\end{proposition}

Then, we are able to again bound $\bar{S}_{\hp,\textrm{BW}}^{\BD}$ for $\BD = \BDU{h}$ or $\BDL{h}$, proving the following theorem, which completes the proof of Part \ref{thm:above-lower} in Theorem~\ref{thm:main}.
Interested readers can look up Section~\ref{App:part2b} for the complete proofs.
\begin{theorem}
\label{thm:lowerbounderr}
Let $\delta > 0$, and let $\omega=(1+\delta)\ln b/{b}$. For all $b\geq \BW$, it is the case that%
\[
\Trel = \Omega\big(n^{d}\big)\text{,\quad where }d=\left(  1+\frac{\ln\left(
\lambda/(1.01\omega b)^2\right)}{2\ln b}\right).
\]
\end{theorem}

%We defer relative lemmas and the rest of the analysis to Section \ref{App:part2b} of the appendix. Together with the upper bound on $\ProbSub{\eta}{A_{\hp}(\eta) = 0}$, we are able to show Part 2b in Theorem~\ref{thm:main}.
% Interested readers can look up Theorem~\ref{thm:lowerbounderr} in Section~\ref{App:part2b} of the appendix.

\subsection{Proofs.}
\label{App:part2b}
\begin{proof}[Proof of Lemma \ref{lem:uni2}]
Let $\tg_{i,1}$ denote $\ProbSub{\eta\sim\muhbd{i}{\BDU{i}}}{A_{\hp}(\eta) = 0}$ and let $\tg_{i,2}$ denote  $\ProbSub{\eta\sim\muhbd{i}{\BDL{i}}}{A_{\hp}(\eta) = 0}$. Let $\bar{t}_{i} = b - t_{i}$ and $\bar{s}_{i} = b - s_{i}$ for simplicity.
Now, recall the recurrences given in Lemma \ref{lem:nonuni2}, which in this case take the form
\begin{align*}
\tg_{i+1,1} &  =\ProbSub{\eta\sim\muhbd{i}{\BDU{i}}}{\eta(r) = 1}
\left( 1-
   \left( 1-
      \tg_{i-1,1}^{t_{i}}
      \tg_{i-1,2}^{\bar{t}_{i}}
   \right)^{t_{i+1}}
   \left( 1-
      \tg_{i-1,1}^{s_{i}}
      \tg_{i-1,2}^{\bar{s}_{i}}
   \right)^{\bar{t}_{i+1}}
\right)\\
&+ \ProbSub{\eta\sim\muhbd{i}{\BDU{i}}}{\eta(r) = 0}
   \left( 1-
      \tg_{i,1}^{t_{i+1}}
      \tg_{i,2}^{\bar{t}_{i+1}}
   \right) \text{,}\\
\tg_{i+1,2} & = \ProbSub{\eta\sim\muhbd{i}{\BDL{i}}}{\eta(r) = 1}
\left( 1-
   \left( 1-
      \tg_{i-1,1}^{t_{i}}
      \tg_{i-1,2}^{\bar{t}_{i}}
   \right)^{s_{i+1}}
   \left( 1-
      \tg_{i-1,1}^{s_{i}}
      \tg_{i-1,2}^{\bar{s}_{i}}
   \right)^{\bar{s}_{i+1}}
\right)\\
&+ \ProbSub{\eta\sim\muhbd{i}{\BDL{i}}}{\eta(r) = 0}
   \left( 1-
      \tg_{i,1}^{s_{i+1}}
      \tg_{i,2}^{\bar{s}_{i+1}}
   \right).
\end{align*}
And the base cases are:
\begin{align*}
\tg_{\hp,1} &= \ProbSub{\eta\sim\muhbd{\hp}{\BDU{\hp}}}{\eta(r) = 0}  \text{,}\quad \tg_{\hp,2} = \ProbSub{\eta\sim\muhbd{\hp}{\BDL{\hp}}}{\eta(r) = 0} \text{,}\\
\tg_{\hp+1,1} & =\ProbSub{\eta\sim\muhbd{\hp+1}{\BDU{\hp+1}}}{\eta(r) = 0}
   \left( 1-
      \tg_{\hp,1}^{t_{\hp+1}}
      \tg_{\hp,2}^{\bar{t}_{\hp+1}}
   \right) \text{,}\\
\tg_{\hp+1,2} & =\ProbSub{\eta\sim\muhbd{\hp+1}{\BDL{\hp+1}}}{\eta(r) = 0}
   \left( 1-
      \tg_{\hp,1}^{s_{\hp+1}}
      \tg_{\hp,2}^{\bar{s}_{\hp+1}}
   \right) \text{.}\\
\end{align*}
Our purpose now, is to show by induction that $\tg_{i,1}\text{, } \tg_{i,2} \leq \fh$ for all $i \geq \hp$. The base case is simple:
\[
\tg_{\hp+1,1} \leq \tg_{\hp,1} \leq  \ProbSub{\eta\sim\muhbd{\hp}{\BDU{\hp}}}{\eta(r) = 0} \leq \frac{1}{1+\omega} \exp\left(\frac{1.01(\omega b)^2}{\lambda}\right),
\]
and the last term is less or equal to $\fh$ for $b \geq \BW$ [Lemma \ref{lem:b0h}]. Similarly, it is the case that $\tg_{\hp+1,2} \leq \tg_{\hp,2} \leq \fh$. Now, in general, assuming the inductive hypothesis, we get from the above recurrence, that
\begin{align*}
\tg_{i+1,1}^{b}
& \leq \left[  \frac{\omega}{1+\omega}\left(  1-\left(  1-\frac{1.01 \omega}{\lambda}\right)  ^{b}\right)
+\frac{1}{1+\omega}  \right]  ^{b}\exp\left(\frac{1.01(\omega b)^2}{\lambda}\right) \text{,}
\end{align*}
where the first term, just as deducted in the proof of Lemma \ref{lem:nonuni2}, is dominated by  $\frac{\exp(1.01(\omega b)^2/\lambda)}{(1+\omega)^b}$. Now, the hypothesis follows for $\tg_{i+1,1}$ (and seemingly for $\tg_{i+1,2}$), due to the fact that  $\exp(2(1.01)(\omega b)^2/\lambda)\leq 1.01 $ for $b\geq\BW$, proving the induction.
\end{proof}

\begin{proof}[Proof of Proposition \ref{prop:eff2}]
Notice the formulas,
\begin{align*}
\muhbd{h}{\BDU{h}}(A(\sigma)=0:\sigma(r)=0) &  =1-\left[
\muhbd{h-1}{\BDU{h-1}}(A(\sigma)=0)\right]  ^{t_{h}}\left[
\muhbd{h-1}{\BDL{h-1}}(A(\sigma)=0)\right]  ^{\bar{t}_{h}}\\
\muhbd{h}{\BDU{h}}(A(\sigma)=1:\sigma(r)=1) &  =\left[  1-\left(
\muhbd{h-2}{\BDU{h-2}}(A(\sigma)=0)\right)  ^{t_{h-1}}\left(
\muhbd{h-2}{\BDL{h-2}}(A(\sigma)=0)\right)  ^{\bar{t}_{h-1}}\right]  ^{t_{h}%
}\\
&  \cdot\left[  1-\left(  \muhbd{h-2}{\BDU{h-2}}(A(\sigma)=0)\right)
^{s_{h-1}}\left(  \muhbd{h-2}{\BDL{h-2}}(A(\sigma)=0)\right)  ^{\bar{s}_{h-1}%
}\right]  ^{\bar{t}_{h}},
\end{align*}
whith similar expressions for $\muhbd{h}{\BDL{h}}\left(  \cdot\right)$. Now,
from these recurrences and the bounds stated in Lemma \ref{lem:uni2}, we deduce that%

\begin{align*}
\muhbd{h}{\BDU{h}}(A(\sigma)=0,\sigma(r)=0)&-\muhbd{h}{\BDU{h}%
}(A(\sigma)=0)\muhbd{h}{\BDU{h}}(\sigma(r)=0) \\
& = \Omega\left(   1-\frac
{1.01}{(1+\omega)^b}-\frac{1.01^{1/b}}{1+\omega}
\right)  \gg0
\end{align*}
With the same result holding for $\muhbd{h}{\BDL{h}}( \cdot) $.
\end{proof}

\begin{proof}[Proof of Theorem \ref{thm:lowerbounderr}]
It goes along the lines of the proof of Theorem \ref{th:slownonunif}. Here, we take $\hp$ as $\hp(\lambda,b)$, as in Lemma \ref{lem:uni2}. Now, due to Lemma \ref{lem:uni1}, we have that
\[
\bar{S}^{\BD}_{\hp,\textrm{BW}} = O\left( \Exp{  \left(  \frac{1.01\omega(1+\omega)
}{\lambda}\right)^ { \#\left\{  i:\sigma(u_{i}) =0\right\} } } \right),
\]
and using the second statement of Lemma \ref{lem:asympunn}, we have that
\[
\Exp{\left(  \frac{1.01\omega(1+\omega)
}{\lambda}\right)^ { \#\left\{  i:\sigma(u_{i}) =0\right\} } } \lesssim \left(
\frac{1.01\omega\exp\left(\frac{1.01(\omega b)^2}{\lambda}\right) }{2\lambda}
\left[
1+\sqrt{1+\frac{4\lambda}{1.01\exp\left(\frac{1.01(\omega b)^2}{\lambda}\right)}}
\right]
\right)^{h-\hp},
\]
where the previous term is dominated by $\left(1.01\frac{\omega}{\lambda^{1/2}}\right)^{h-\hp}$. \ricardo{sentence edited} And just as in the proof of Theorem \ref{th:slownonunif},
\[
\bar{S}_A =O \left(\left[\frac{1.01\omega}{\lambda^{1/2}}\right]^h\right)
=O\left(n^{-\left[
1+\frac{\ln\left(\lambda/(1.01\omega b)^2\right)}{2\ln b}
\right]}\right)\text{.}
\]

Now, from Proposition \ref{prop:eff2}, the BW\ algorithm is effective for $\omega>(1+\delta)\ln{b}/{b}$, $b>b_0(\delta)$, therefore Theorem \ref{thm:sensitivity2} applies. The conclusion
follows.
\end{proof}
\section{{\protect\large Upper Bounds of the Relaxation Time via the Coupling
Method}}
\label{sec:upper-bound}
\smallskip
We will use the coupling technique in some of our proofs which upper bound the
relaxation time and the mixing time.
The mixing time $\Tmix$ for the Glauber dynamics is defined as the number
of steps, from the worst initial state, to reach within variation distance
$\le 1/2\eee$ of the stationary distribution.
It is an elementary fact that the mixing time gives a good upper bound on the relaxation time (see, e.g., \cite{LPW} for the following bound),
we will use this fact in our upper bound proofs:
\begin{equation}
\label{eqn:TmixTrel}
\Trel \le \Tmix + 1.
\end{equation}

 Given two copies $(X_t)$ and $(Y_t)$ of the Glauber dynamics,
 a coupling is a joint process $(X_t,Y_t)$ such that the evolution of each component
 viewed in isolation is identical to the Glauber dynamics, see \cite{LPW}
 for more background on the coupling technique.
  The Coupling Lemma \cite{Aldous} (c.f., \cite[Theorem 5.2]{LPW})
  guarantees that if,
 there is a coupling and time
 $t >0$, so that for every pair $(X_0, Y_0)$ of initial states,  $\ProbCond{X_t \neq Y_t}{X_0,Y_0} \le 1/2\eee$  under the coupling, then $\Tmix \le t$.

Before we show the main idea of proving the upper bound, let us introduction some notations for this section first. Let $\tau_{\rho}$ be the relaxation time of the following Glauber dynamics of the hard-core model
on the star graph $G$, where $\rho$ is a $b$ dimensional vector. The dynamics is defined as: if the root is chosen, it
will be occupied with probability $\lambda/(1+\lambda)$ if no leaf is
occupied; if the leaf $i$ is chosen, it will be occupied with probability
$\rho_{i}$ if the root is unoccupied. Therefore $\tau := \max_{\rho}\{\tau_{\rho}\}$ is defined as the
worst case relaxation time over all the possible choice of $\rho$.
Following the same block dynamics strategy \cite{Martinelli-lecturenotes} as in Section 2.3 of the paper by \cite{BKMP} (See also \cite{Molloy} and \cite{TVVY}),
it is not hard to show that the relaxation time
of the above Glauber dynamics is exactly the same as that of the natural block dynamics which updates the configurations of a whole subtree of the root in one step, and hence the
following lemma holds.
\begin{lemma}
\label{lem:block}
The relaxation time $T_{rel}$ of the Glauber dynamics of the hard-core model on the complete
tree of height $H$ is upper bounded by $\tau^{H}$ for any boundary condition on the leaves.
\end{lemma}
Note that, the relaxation time is quite sensitive with respect to the boundary conditions. Especially, in the paper by \cite{MSW-soda}, they show that when the boundary conditions are even (or odd) meaning that occupied all the leaves when the height is even (odd) and unoccupied all the leaves when the height is odd (even), the mixing time is actually $O(n\ln n)$. In this paper we are dealing with any kind of boundary condition, and in the lower bound part, we show the boundary conditions that slow down the Glauber dynamics actually exist. The lower bound on the relaxation time for the Glauber dynamics under that boundary conditions roughly matches up with the upper bound here.

We need to bound the relaxation time $\tau$ of the Glauber dynamics on star graphs for different cases with respect to $\rho$.
To do this, we first use the so-called maximum one step coupling to bound the mixing time and then by Eq.\eqref{eqn:TmixTrel} we get an upper bound on the relaxation time $\tau$.
The maximal one-step coupling, originally studied for colorings by Jerrum \cite{Jerrum} gives a way to upper bound the mixing time of the Glauber dynamics on general graphs.
The coupling $(X_t,Y_t)$ of the two chains is done by choosing the same random vertex $v$ for changing the states at step $t$ and maximizing the probability of the two chains choosing the same update for the state of $v$.
Thus, if none of neighbors of $v$ is occupied in neither $X_t$ and $Y_t$, $v$ will be occupied/unoccupied in both chains at time $t+1$ with the correct marginal probability. In all other cases, the update choices for $X_{t+1}(v)$ and $Y_{t+1}(v)$ are coupled arbitrarily.
By analyzing the maximum one-step coupling, we will show a series of lemmas for different $\rho$, which give the upper bound on the relaxation time $\tau$.
\begin{lemma}
\label{lemma:up1}
If $\sum \rho_{i} \le 4\ln\ln b$ and for all $i$, $\rho_{i} \le 1-1/\ln b$, then $\tau = O(b^{1+o(1)})$, where $o(1)$ is a term that goes to zero as $b$ goes to infinity.
\end{lemma}
\begin{lemma}
\label{lemma:up2}
If $\sum \rho_{i} \ge 4\ln\ln b$, then $\tau = O( (\lambda+1) b\ln b)$.
\end{lemma}
\begin{lemma}
\label{lemma:up3}
If there exists an $i$ such that $\rho_i > 1-1/\ln b$, then $\tau = O((\lambda+1) b\ln b)$.
\end{lemma}
The intuition behind is that the typical behaviors of the coupling chain change with respect to the leaves marginal probability $\rho$. When $\rho$ are all tiny (Lemma \ref{lemma:up1}), the coupling chain has a large chance to have all the leaves unoccupied in both $(X_t)$ and $(Y_t)$ and hence can be coupled at the state where the root is occupied in both chains with a good probability; While when the sum of $\rho$ is big (Lemma \ref{lemma:up2} and \ref{lemma:up3}) suggesting that there is a good chance that one of the leaves is occupied always, the coupling chain is easier to get coupled once the root is not occupied in both chains.
We delay the proofs of these lemmas in Section \ref{section:up_proofs}. By applying Lemma \ref{lem:block} to Lemma \ref{lemma:up1}, \ref{lemma:up2} and \ref{lemma:up3}, we get the conclusion that the relaxation time is upper bounded by $O(n^{1+\ln(\lambda+1)/\ln b+o_b(1)})$ for any $\lambda > 0$. Recall that the relationship between $\omega$ and $\lambda$ is $\lambda = \omega (1+\omega)^b$ and in this paper we are mainly interested in the cases when $\omega = (1+\delta)\ln b/b$ for any constant $\delta>-1$. Hence, in terms of $\omega$, if $-1 < \delta \le 0$ the relaxation time is upper bounded by $n^{1+o_b(1)}$  and if $\delta > 0$ the relaxation time is upper bounded by $n^{1+\delta+o_b(1)}$, which proves Theorem \ref{thm:main}.

\subsection{Proofs.}
\label{section:up_proofs}
With out lose of generality, we assume that the root is unoccupied in $X_0$ and occupied in $Y_0$. Readers will see later in the proofs that these are indeed the worst cases scenarios.
We prove the lemmas by analyzing the coupling in rounds, where each round consists of $T:=20 b\ln{b}$ steps.
The following analysis says
that if in each round we have a good probability of coalescing (i.e., achieving $X_t=Y_t$), then we will have a good upper bound on the mixing time (and hence the relaxation time).
Suppose that in the coupling process, starting from any pair of initial states, after time $T$, with probability $P_T$ we have $X_T = Y_T$.
By repeatedly applying this assumption $i$ many times we have, for all $(x_0,y_0)$,
\[
\ProbCond{X_{2iT} \neq  Y_{2iT}}{X_0 = x_0, Y_0 = y_0} \leq
(1-P_T)^{2i}\leq 1/2\eee
\]
for $i=1/P_T$.
Therefore, by applying the Coupling Lemma, mentioned in Section~\ref{sec:background}, the mixing time is $O(b\ln b /P_T)$. Now, the only thing left is to lower bound $P_T$ for different $\rho$ assuming that the root $r\notin X_0$ and $r\in Y_0$. Let us define for any configuration $X_t \in \Omega(G)$, $\|X_t\|:= |\{ \ell \in V(G): deg(\ell)=1 \textrm{ and } \ell \in X_t\}|$, the number of leaves of the star graph $G$ that are in the independent set.
\begin{proof}[Proof of Lemma \ref{lemma:up1}.]
In this case, we will try to couple all of the leaves to unoccupied and then couple the root. Let $T_0$ be the last time the root is chosen to change the state in the $T$ steps running of the coupling chain. It is a simple fact from the coupon collector problem that with high probability, $T_0 > 10b\ln b$ and all the leaves have been chosen at least once before $T_0$. Note that throughout this section ``with high probability'' means the probability goes to $1$ as $b$ goes to infinity. Conditioning on this, consider the following process $(B_t)$:
\begin{itemize}
\item $B_t \subseteq V(G)$;
\item For each time $t$ the process $(B_t)$ chooses a leaf $v_i$ to update: $B_{t+1}= B_{t}\cup {v_i}$ with probability $\rho_i$ and $B_{t+1}= B_{t}\setminus {\{v_i\}}$ with probability $1-\rho_i$;
\item The process will ignore any update of the root and the state of the root at any time.
\end{itemize}
By the conditioning we assume, we know that after $10b\ln b$ steps, every leaf has been updated at least once. Therefore the chance that all of the leaves are not in $B_t$ (i.e. $(B_t)$ is an empty set) is \[
\prod_{i=1}^{b} (1-\rho_i) \ge \prod_{i=1}^{b}e^{-\rho_i \cdot \ln\ln b}\ge e^{-4\ln^2\ln b},
\]
where $1-x > e^{-x \cdot \ln\ln b}$ holds when $x < 1 - 1/\ln b$. There is a straightforward coupling of $(X_t)$ and $(B_t)$ such that for any $t\ge 0$, $\|X_t\| \le |B_t|$.
And also, by the natural of the maximum one-step couplings with the assumption $r\notin X_0$ and $r\in Y_0$, for any $t\ge 0$, $\|X_t\| \ge \|Y_t\|$. Therefore, we can conclude that with probability at least $ e^{-4\ln^2\ln b}$,
at time $T_0 - 1$, all the leaves are unoccupied in both $X$ and $Y$, and hence with probability $1$, when the coupling chain chooses the root to update at time $T_0$, we have $X_{T_0} = Y_{T_0}$. In conclusion, we show that
with probability at least $e^{-4\ln^2\ln b}$, we have $X_T = Y_T$, which implies that the mixing time and hence the relaxation time is $O(b\ln b \cdot  e^{4\ln^2\ln b}) = O(b^{1+o_b(1)})$, where $o_b(1)$ is a $O(\ln^2\ln b)/\ln b$ function.
%Then we can conclude that in $b\ln b \cdot (\ln b)^{\ln\ln b} = b^{1+o_b(1)}$ steps, with at least constant probability, $B$ is an empty set.
\end{proof}
\begin{proof}[Proof of Lemma \ref{lemma:up2}.]
In this case, we want to first unoccupied the root in both chains and then couple the leaves. The key observation is that once the root is not occupied, it will not be occupied again since there is a good chance that one of the leaves is always occupied and hence ``blocking'' the root from being occupied.
First of all, it is easy to see that in the first $b$ steps, the coupling chain has a positive probability of choosing the root to change the state and at that moment, the chance of the root being unoccupied in both chains is at least $1/(\lambda+1)$. We denote this as event $\E_1$. Next, we want that in the following $b$ steps, with positive probability, the root will not be chosen to change the state, therefore the chain chooses $b$ leaves to change the states during these $b$ steps. Let the set of leaves that are chosen during these $b$ steps be $S$. We want to argue that with positive probability, $\sum_{i\in S} \rho_i > c_0 \ln\ln b$ for some constant $c_0 > 1$ so that $S$ is a set of good representatives of all the leaves. Let this be the event $\E_2$. By using the fact that for each leaf, the indicator random variable of whether the leaf is chosen or not during these $b$ steps are negatively associated to each other (c.f., Theorem 14 in \cite{DR}), we can still use the following Hoeffding bound (Theorem 2 in \cite{HF}) by Proposition 7 in \cite{DR}.
\begin{theorem}
Let $X_1,\dots,X_n$ be independent (or negatively associated) random variables with $a_i \le X_i \le b_i$. Let $X=\sum_i X_i$ and $\mu = E[X]$. Then the following inequality holds
\[
\Prob{|X-\mu|\ge t}\le 2\exp(-\frac{2t^2}{\sum_{i=1}^{n}(a_i - b_i)^2}).
\]
\end{theorem}
In our case, $X_i = \rho_i$ if the leaf $i$ is chosen in $b$ steps and otherwise zero.
Therefore, since in our case $\sum_{i=1}^{n}(a_i - b_i)^2\le (\sum_{i=1}^{n}(b_i - a_i))^2 = (4\ln\ln b)^2$, $\mu \ge 4(1-1/e)\ln\ln b$ and $t = \delta \mu$ for some properly chosen $\delta > 0$, it is straightforward to show that with positive probability, the event $\E_2$ happens. Note that the events $\E_1$ and $\E_2$ are actually independent. And from now on, we are conditioning on $\E_1$ and $\E_2$.

For each $t > 2b$, let $A_t:= \{\ell \in S: \ell \in X_t \textrm{ and } \ell \in Y_t\}$ be the set of leaves in $S$ such that they are in the independent set at time $t$ in both $X_t$ and $Y_t$.
Let the stopping time $T_s$ be the first time $t > 2b$ such that $A_t = \emptyset$ and the root is chosen. If we can show that $\Prob{T_s < 2b\ln b} < 1/2$, then with positive probability, the root is always blocked from changing to occupied since $A_t \neq \emptyset$. Hence once the chain chooses all the leaves during $2b\ln b-2b$ steps, we will have $X_T = Y_T$. Finally, because the events $\E_1$ and $\E_2$ happen with probability $\Omega(1/(\lambda+1))$, we can conclude that with probability $\Omega(1/(\lambda+1))$, $X_T = Y_T$.

In order to show that $\Prob{T_s < 2b\ln b} < 1/2$,
we will use the same coupling strategy as in the last proof. Let $(C_t)$ be a stochastic process that update the states of $|S|$ many independent vertices in the following way:
\begin{itemize}
\item The process starts at time $t = 2b+1$ with a random starting configuration;
\item For any $t$, $C_t \subseteq S$;
\item For each time $t$ the process $(C_t)$ chooses a vertex $v\in V(G)$ randomly;
\item If $v = v_i \in S$ then apply the following update rule: $C_{t+1}= C_{t}\cup {v_i}$
      with probability $\rho_i$ and $B_{t+1}= B_{t}\setminus {\{v_i\}}$ with probability $1-\rho_i$.
\end{itemize}
Let the stopping time $T_c$ be the first time $t>2b$ such that $C_t = \emptyset$ and the root is chosen. We want to couple $(C_t)$ and $(X_t,Y_t)$ in such a way that
$T_s > T_c$ for all runs and then if we can show that $\Prob{T_c < 2b\ln b} < 1/2$, by the coupling, we are done.

In fact, it is easy to show that $\Prob{T_c < 2b\ln b} < 1/2$ (any positive constant will do) by using
union bound combining with the following two facts: for any fix time $t > 2b$, the probability that $C_t = \emptyset$ is upper bounded by $(\ln b)^{-c_0}$ with $c_0 > 1$; with high probability the
root will be chosen at most $O(\ln b)$ times during $T$ steps. Now we just need a valid coupling.

We couple $(C_t)$ with the coupling chain $(X_t,Y_t)$ from time $t = 2b+1$ in the following way:
\begin{itemize}
\item Initially, for $t=2b$, $v\in C_t$ if and only if $v\in A_t$.
\item For any time $t$, both of them pick the same vertex $v$ to update and if $v\in S$, $v\in C_t$ if and only if $v\in A_t$.
\end{itemize}
It is easy to verify that conditioning on the events $\E_1$ and $\E_2$,
 this coupling is valid and $T_s > T_c$, which completes the proof.
\end{proof}
\begin{proof}[Proof of Lemma \ref{lemma:up3}.]
This case is the same as the last one. However, we can just use the leaf $i$ to block the root since it has a very high probability of being occupied, i.e., for the probability bound concerning $T_s$, it is enough to bound the probability that at a given time $t$, the leaf $i$ is occupied.
\end{proof}

\section*{Acknowledgments}
We are grateful to Prasad Tetali for many helpful discussions.

\end{document}